\documentclass{proc-l}
\usepackage{xcolor}

\newtheorem{theorem}{Theorem}[section]

\theoremstyle{definition}
\newtheorem{definition}[theorem]{Definition}

\theoremstyle{remark}
\newtheorem{remark}[theorem]{Remark}
\numberwithin{equation}{section}
\newcommand{\typeRPermutation}{\text{Type R permutation}}
\newcommand{\typeRlambdaPermutation}{\text{type R $\lambda$-permutation}}

\begin{document}
\title{Type R $\lambda$-Permutation Approach to Velleman's Open Problem}

\author{
  Polymath Jr. 2020 Collaboration
}
\footnotetext[1]{\hangindent=2em \hangafter=1
Student authors and affiliations, ordered by last name:
Hadi Hammoud (American University of Beirut),
Andrew D Harsh (Swarthmore College),
Antonio Marino (Carleton College),
Assaf Marzan (Hebrew University of Jerusalem),
Daniil Nikolievich Shaposhnikov (San Francisco State University),
Kealan Vasquez (Ave Maria University),
Hui Xiao (Oklahoma State University)}
\footnotetext[2]{\hangindent=2em \hangafter=1 
Mentor of the Polymath Jr. 2020 student authors: Yunus Zeytuncu (Department of Mathematics and Statistics, University of Michigan–Dearborn) }
\footnotetext[3]{\hangindent=2em \hangafter=1
We thank graduate student mentors Cody Stockdale and Nathan Wagner for their support during the project.}

\begin{abstract}
Previously, mathematicians Steven Krantz and Jeffery McNeal studied a type of positive numbers permutation called $\lambda$-permutation. This type of permutation, when applied to the index of terms of a series, is defined to be both convergence-preserving and ``fixing'' at least one divergent series, that is, rearranging the terms of any convergent series will result in a convergent series, while rearranging the terms of some divergent series will result in a convergent series. In general, if a divergent series can be fixed to converge in some way (it does not need to be by $\lambda$-permutation), it is called a ``conditionally divergent series''. In 2006, another mathematician Daniel Velleman raised an open problem related to $\lambda$-permutation: for a conditionally divergent series $\sum_{n=0}^{\infty}a_n,n\in \mathbb{N},a_n\in 
\mathbb{R}$, let $S=\{L \in \mathbb{R} \colon L = \sum_{n=0}^{\infty}{a_{\sigma\left(n\right)}}$ $\text{for some } \lambda\text{-permutation } \sigma\}$, can $S$ ever be something between $\emptyset$ and $\mathbb{R}$? This paper is devoted to partially answering this open problem by considering a subset of $\lambda$-permutation constraint by how we can permute, named type R $\lambda$-permutation. Then we answer the analogous question about a subset of S with respect to type R $\lambda$-permutation, named $Z_{R}=\{L \in \mathbb{R} \colon L = \sum_{n=0}^{\infty}{a_{\sigma\left(n\right)}}$ $\text{for some type R } \lambda \text{-permutation } \sigma\}$. We show that $Z_R$ is either $\emptyset$, a singleton or $\mathbb{R}$. We also provide sufficient conditions on the conditionally divergent series $\sum_{n=0}^{\infty}a_n$ for $Z_R$ to be a singleton or $\mathbb{R}$, by introducing a "substantial property" on the series.
\end{abstract}
\maketitle

\section{Introduction}
Steven Krantz and Jeffery McNeal studied a type of positive numbers permutation called $\lambda$-permutation in their paper \cite{A}. This type of permutation, when applied to the index of terms of series, is defined to be both convergence-preserving and ``fixing'' at least one divergent series, that is, rearranging the terms of any convergent series will result in convergent series while rearranging the terms of some divergent series will result in convergent series. Daniel Velleman \cite{C} refers to divergent series that can be fixed by some permutations (it does not need to be by $\lambda$-permutation) as ``conditionally divergent series''. 

\subsection{Bounded block number}
It was also proved \cite{C} that a permutation of positive integers $\sigma$ is convergence-preserving if and only if it has a bounded ``block number sequence''.
A formal definition of block number sequence can be found in Velleman's paper. Here we illustrate it with an example to help readers develop intuition: Let $\sigma$ be a permutation of positive numbers and the first four permutation steps are $\sigma\left(1\right)=1$,$\sigma\left(2\right)=3$,$\sigma\left(3\right)=4$,$\sigma\left(4\right)=2$. Then the block number sequence of $\sigma$ for the first 4 permutation steps is $(1, 2, 2, 1)$. Here is how the sequence was obtained:
$\sigma\left(1\right)=1 => 1       => 1$ block,
$\sigma\left(2\right)=3 => 1 \text{ } 3   => 2$ blocks since 1 and 3 are not connected,
$\sigma\left(3\right)=4 => 1 \text{ } 34 => 2$ blocks since 1 and 3 are not connected but 3 and 4 are connected,
$\sigma\left(4\right)=2 => 1234 => 1$ block since 1, 2, 3 and 4 are connected.  

The maximum number of this sequence so far is 2. As permutation goes, this maximum may increase or stay the same. If the maximum number is bounded as the sequence approaches infinity, then $\sigma$ has a bounded block number sequence. For convenience, we refer to any upper bound of the maximum number of a bounded block number sequence as a bounded block number in our paper.

Together, given a conditionally divergent series $S\ =\sum_{n=0}^{\infty}a_n,n\in \mathbb{N},a_n\in \mathbb{R}$ and a positive integers permutation $\sigma$ that fixes $S$ to some value $r\in \mathbb{R}$, $\sigma$ is a $\lambda$-permutation of S if and only if it has some bounded block number.

\subsection{Conditionally divergent series}
Conditionally divergent series are divergent series that can be rearranged to converge. For a given series, there can be many permutations to achieve this. The type of permutation involved in a standard textbook proof of the Riemann rearrangement theorem \cite{B} is one that rearranges any conditionally divergent series to converge. 

Riemann rearrangement theorem states that a conditionally convergent series of real numbers $\sum_{n=0}^{\infty}a_n,n\in \mathbb{N},a_n\in \mathbb{R}$ can be arranged to converge to an arbitrary real number $r$or diverge by some permutation. A textbook proof \cite{B} for the converging statement usually uses a permutation constructed in the following procedure: if the finite sum of the terms chosen from the original conditionally convergent series $S$ so far is less than or equal to $r$, then choose the first unused positive term of S to be the next term of the permuted series $S^\sigma$ by $\sigma$; if the finite sum is greater than $r$, then choose the first unused negative term of $S$.

The essence of the proof to why such a permutation can rearrange any conditionally convergent series to converge to any value utilizes two properties of conditionally convergent series: (a) The partial sum of all the positive terms or all the negative terms in the conditionally convergent series diverges. (b) The terms of the conditionally convergent series converge to 0. We will leave it to the reader to prove that conditionally divergent series also satisfy these two properties. Therefore, we can apply the same procedure to construct a permutation that fixes a conditionally divergent series to an arbitrary real number $r$ as in the proof of the Riemann rearrangement theorem.

In fact, the property ``in-order selection of same-sign elements'' of this permutation motivated us to restrict the $\lambda$-permutation to those with this property. We name this restricted type ``type R $\lambda$ -permutation''. With type R $\lambda$-permutation, our paper makes progress by considering a subset of $S$ in Velleman's open problem, $Z_{R}=\{L \in \mathbb{R} \colon L = \sum_{n=0}^{\infty}{a_{\sigma\left(n\right)}}$ $\text{for some type R } \lambda \text{-permutation } \sigma\}$, and we reach the conclusion that $Z_R$ is either $\emptyset$, a singleton or $\mathbb{R}$. We also prove sufficient conditions on the conditionally divergent series$ \sum_{n=0}^{\infty}a_n$ for $Z_R$ to be a singleton or $\mathbb{R}$. 

\section{Definition and Notation}
\begin{definition}[Positive/Negative block]

For $S = \sum_{n=0}^{\infty}a_n,n\in \mathbb{N},a_n\in \mathbb{R}$, define $P_S=\{n \in \mathbb{N}:a_n > 0 \}$ and $N_S=\{n \in \mathbb{N}: a_n \le 0 \}$ as sets of positive and negative indices with respect to $S$. We simply write $P, N$ when it is clear from the context which series they are with respect to.

Velleman defined \cite{C} the following interval notation to better describe permutations: for positive integers $c,d$ and $c\le d$ 
\[
\left[c,d\right]_\mathbb{Z}=\{x\in \mathbb{Z}^+:c \le x \le d\}
\]
Following this interval notation, present $P, N$ as a union of intervals:
\[
    P = 
    [{p_1},{q_1}]_{\mathbb{Z}} \cup [{p_2},{q_2}]_{\mathbb{Z}} \cup [{p_3},{q_3}]_{\mathbb{Z}} \cdots
\]
\[
     N = [{n_1},{m_1}]_{\mathbb{Z}} \cup [{n_2},{m_2}]_{\mathbb{Z}} \cup [{n_3},{m_3}]_{\mathbb{Z}} \cdots
\]
Define $P_i = [{p_i},{q_i}]_{\mathbb{Z}}$ and refer to it as the $i$-th positive block of $S$ and similarly $N_i = [{n_i},{m_i}]_{\mathbb{Z}}$ as the $i$-th negative block of $S$. Also, denote $S_{q_i}$(or $S_{m_i}$) as the finite sum of the series $S$ up to the last term $q_i$ in ${P_i}$(or $m_i$ in ${N_i}$). Similar notation applies to the permuted series $S^{\sigma}_{q_i}$(or $S^{\sigma}_{m_i}$)
\end{definition}

\begin{definition}[Summation over {same-sign block(s)}]
As we will be looking at summation over block(s) of same sign frequently, it will be helpful to define the following notation of summation over positive blocks from $P_i$ to $P_j$ ($i \leq j$) for $S = \sum_{n=0}^{\infty} a_n, n \in \mathbb{N}, a_n \in \mathbb{R}$:
\[
    S_{[P_i, P_j]} := \sum_{n \in \bigcup\limits_{k=i}^{j} P_k} a_n
\]
Similarly, we have summation over negative blocks from $N_i$ to $N_j$ as \[
    S_{[N_i, N_j]} := \sum_{n \in \bigcup\limits_{k=i}^{j} N_k} a_n
\]
\end{definition}

\begin{definition}[\typeRPermutation]
For $S = \sum_{n=0}^{\infty} a_n, n \in \mathbb{N}, a_n \in \mathbb{R}$, $\sigma$ is a type R permutation of $S$ if it preserves the order of terms of same sign in $S$. In other words, if $\sigma(i) < \sigma(j)$ and $a_{\sigma(i)}, a_{\sigma(j)}$ are both positive or both negative, then $i < j$.
\end{definition}
\begin{remark}
The permutation used to prove the Riemann rearrangement theorem described in Section `Conditionally divergent series' is a type R permutation.
\end{remark}

\begin{definition}[Type R $\lambda$-permutation]
For a conditionally divergent series $S = \sum_{n=0}^{\infty} a_n, n \in \mathbb{N}, a_n \in \mathbb{R}$ and $\sigma$ that fixes $S$, $\sigma$ is a type R $\lambda$-permutation of $S$ if it is both type R and has a bounded block number.
\begin{remark}
The advantage of restricting $\lambda$-permutation to type R $\lambda$-permutation defined is that, by imposing an order in how we permute terms, we are able to relate the partial sum of the permuted series with the one of the original series. 

To illustrate, let $S = \sum_{n=0}^{\infty} a_n, n \in \mathbb{N}, a_n \in \mathbb{R}$ be the original series and $S^{\sigma} = \sum_{n=0}^{\infty} a_{\sigma(n)}$, $a_{\sigma(n)}\in \mathbb{R}$ be the permuted series, where $\sigma$ is a type R $\lambda$-permutation with bounded block number $C$. Without loss of generality, we may assume that $S$ starts with a negative term. (In fact, we can always assume that $S$ starts with a negative block consisting only of a zero. Therefore, we will assume all the series $S$ in this paper start with a negative block)

Since type R permutation preserves the order of blocks of same sign, we have the following inequality:
\begin{align}
    \forall i \in \mathbb{N}, S_{[N_1, N_i]} + S_{[P_1, P_{i-C}]} \leq S^\sigma_{m_i} \leq  S_{[N_1, N_i]} + S_{[P_1, P_{i+C-1}]}    
\end{align}

To show this, based on bounded block number $C$, by the time when the last index in the negative block $N_i$ was permuted, the last index in the positive block $P_{i-C}$ would have already been permuted. Otherwise, there will be enough gaps between permuted indexes such that they form more than $C$ blocks. This will violate the assumption of a bounded block number $C$. For the same reason, by the time that the last index in the negative block $N_i$ was permuted, the first index in the positive block $P_{i+C}$ should not have been permuted. 

Inequality (2.1) is same as 
\begin{align}
    \forall i \in \mathbb{N}, S_{q_{i-C}} + S_{[N_{i-C+1}, N_i]} \leq S^\sigma_{m_i} \leq  S_{m_{i}} + S_{[P_{i}, P_{i+C-1}]}
\end{align}
Therefore, we relate the partial sum of the permuted series, $S^\sigma_{m_i}$, with the ones of the original series, $S_{q_{i-C}}$ and $S_{m_{i}}$, plus the sum of a finite number of same-sign blocks of the original series $S_{[N_{i-C+1}, N_i]}$ and  $S_{[P_{i}, P_{i+C-1}]}$. 

Inequality (2.1) is the key observation on the relationship between the original and permuted series under type R $\lambda$-permutation, which will be utilized in the proof of our main theorem. This relation also motivates our next property ``substantial property'' defined on finite blocks of the same sign of the original series in order to gain clarity on the value to which the permuted series converges.
\end{remark}
\end{definition}

\begin{definition}[Substantial property]
For $S = \sum_{n=0}^{\infty} a_n, n \in \mathbb{N}, a_n \in \mathbb{R}$, $S$ satisfies the substantial property on positive blocks, denoted as $ST_P$, if there exist some $k \in \mathbb{N}$ and some $\epsilon \in \mathbb{R}^{+}$ such that for some $i_0 \in \mathbb{N}$,  we have $S_{[P_i, P_{i+k}]} \geq \epsilon \text{ for } \forall i>i_0$. Similarly, $S$ satisfies substantial property on negative blocks, denoted as $ST_N$, if there exists some $k \in \mathbb{N}$ and some $\epsilon \in \mathbb{R}^{+}$ such that for some $i_0 \in \mathbb{N}$,  we have $S_{[N_i, N_{i+k}]} \leq -\epsilon \text{ for } \forall i>i_0$ 
\begin{remark}
Substantial property means (a) the sum of a finite number of same-sign blocks is big enough to bound away from 0 (i.e. "substantially large") (2) there are infinite amount of such same-sign blocks (i.e. “substantially many"). 

Velleman provided an example series with the substantial properties \cite{C}: \\ 
$\sum_{k=1}^{\infty} a_k = 1 - 1 - \frac{1}{2} - \frac{1}{3} - \frac{1}{4} + \frac{1}{2} + \frac{1}{3} + ... + \frac{1}{33} - \frac{1}{5} - ...$ where the positive and negative blocks are constructed as following: the first block consists of only positive terms such that the partial sum of the series equals to some $s_1$ where $s_{1} \geq 1$, the second block consists of only negative terms such that the partial sum of the series equals to $-t_1$ where $t_1 \geq s_{1} + 1$, the third block consists of only positive terms such that the partial sum of the series equals to $s_2$ where $s_2 \geq t_1+1$ and so on. Since $S_{\left[P_i,P_i\right]},{|S}_{\left[N_i,N_i\right]}|\geq1$,$\forall i\in \mathbb{N}$,  $S$ satisfies both $ST_P$ and $ST_N$.
\end{remark}
\end{definition}

\section{Main theorem}
\begin{theorem}
Let $S = \sum_{n=0}^{\infty} a_n, n \in \mathbb{N}, a_n \in \mathbb{R}$ be a conditionally divergent series and consider the set $Z_{R} = \{r \in \mathbb{R} : r = S^{\sigma}, \text{ where $\sigma$ is some \typeRlambdaPermutation}\}$. If $S$ cannot be fixed to any value by any \typeRlambdaPermutation, then $Z_{R} = \emptyset$. Otherwise (i.e, $|Z_{R}| \geq 1$), there are two cases:
\begin{enumerate}
\item[(1)]If $S$ satisfies $ST_P$ and $ST_N$, then $Z_{R} =\mathbb{R}$. 
\item[(2)] Otherwise (that is, $S$ violates $ST_P$ or violates $ST_N$), $Z_{R}$ is a singleton. (i.e., $|Z_{R}| = 1$)
\end{enumerate}
\end{theorem}
\begin{proof}
\textbf{To prove (1)}: we adopt the same high-level proof strategy as in Velleman's paper \cite{C}.  That is, to prove that for $\forall r \in \mathbb{R}$, we can construct a 
type R permutation that fixes $S$ to $r$ and prove the permutation is also a $\lambda$ permutation (i.e, has a bounded block number) by showing that the first term of $(i+1)$-th positive block of $S$ has to be permuted by $\sigma$ before the first term of the $(i+C)$-th negative block of S for  $\forall i > i_3$ for some $i_3 \in \mathbb{N}$ for every $r \in \mathbb{R}$. The proof for the symmetric case proceeds similarly.

For $\forall r \in \mathbb{R}$, we construct a type R permutation $\sigma$ that fixes $S$ to $r$ using the following procedure: if the finite sum of the terms chosen so far is less than or  equal  to $r$,  then  we  choose  the  first  unused  positive  term  of $S$ to be the next term of the permuted series $S^{\sigma}$. If the finite sum is greater than $r$, then we choose the first unused negative term.  Clearly $\sigma$ is a type R permutation that fixes $S$ to r. All we left to do is prove that $\sigma$ is also a $\lambda$-permutation, that is, it has a bounded block number.

Since $S$ can be fixed by at least one $\typeRlambdaPermutation$ \text{ }$\sigma_0$ of some bounded block number $C_0$ to some $r_0$, by inequality (2.2), we have 
\begin{align}
\forall i \in \mathbb{N}, S_{q_{i-C_0}} + S_{[N_{i-C_0+1}, N_i]} \leq S^{\sigma_0}_{m_i} \leq  S_{m_{i}} + S_{[P_{i}, P_{i+C_0-1}]}
\end{align}
Since $S^{\sigma_0} = r_0$, we have $S^{\sigma_0}_{m_i} \rightarrow{} r_0, i \rightarrow{} \infty$. Let $\epsilon$ be an arbitrary positive real number and denote $\epsilon_0 = r_0 + \epsilon$. By convergence of series, for some $i_0 \in \mathbb{N}$, we have the following. 
\begin{align}
\forall i > i_0, S_{q_{i-C_0}} + S_{[N_{i-C_0+1}, N_i]} \leq S^{\sigma_0}_{m_i} < r_0 + \epsilon = \epsilon_0
\end{align}
equivalent to
\begin{align}
\forall i > i_0 - C_0 \in \mathbb{N}, S_{q_{i}} + S_{[N_{i+1}, N_{i+C_0}]}  < r_0 + \epsilon = \epsilon_0
\end{align}
Since $S$ satisfies $ST_N$, there exist some $k_1 \in \mathbb{N}$ and some $\epsilon_1 \in \mathbb{R}^{+}$ such that for some $i_1 \in \mathbb{N}$,  we have $S_{[N_i, N_{i+k_1}]} \leq -\epsilon_1 \text{ for } \forall i>i_1$. 
By the Archimedean property of $\mathbb{R}$, there exists $M \in \mathbb{N}$ such that $\epsilon_0 - M * \epsilon_1 < r$. The existence of such $M$ then implies the following:
\begin{equation}
\begin{aligned} 
\forall i > \text{max}(i_0 - C_0, i_1) \in \mathbb{N}, S_{q_{i}} + S_{[N_{i+1}, N_{i+C_0}]}
& + S_{[N_{i+C_0+1}, N_{i+C_0+1+k_1}]} \\
& + S_{[N_{i+C_0+2+k_1}, N_{i+C_0+2+2*k_1}]} \\
& + S_{[N_{i+C_0+3+2*k_1}, N_{i+C_0+3+3*k_1}]} \\
& + ... \\
& + S_{[N_{i+C_0+M+(M-1)*k_1}, N_{i+C_0+M+M*k_1}]} \\
& = S_{q_{i}} + S_{[N_{i+1}, N_{i+C_0+M+M*k_1}]} \\
& < \epsilon_0 - M * \epsilon_1  < r
\end{aligned}
\end{equation}
Intuitively speaking, (3.4) is about combining $M$ consecutive groups of $(k_1+1)$ negative blocks to create a larger cluster of consecutive negative blocks that add up to a sufficiently large negative effect such that the partial sum is less than $r$

Let $C = C_0 + M + M * k_1 + 1$. Suppose that we have the first term of the $(i+C)$ -th negative block of $S$ being permuted by $\sigma$ before the $(i+1)$ -th positive block and this occurs infinitely often for $i \in \mathbb{N}$. Since $\sigma$ is type R and we have selected a negative term (that is, the first term of the $(i+C)$-th negative block), then we know the finite sum of $S^{\sigma}$ before selecting that negative term, denoted as $S^{\sigma}_n$ where $n$ is the index of the last permuted element before permuting that negative term, which is greater than $r$.
\begin{align}
    S^{\sigma}_n > r
\end{align}
Also, since $\sigma$ is type R, which means that we permute blocks of the same sign in order, then we have
\begin{align}
   S^{\sigma}_n \leq S_{[N_{1}, N_{i+C-1}]} + S_{[P_{1}, P_{i]}}
\end{align}
Since all negative blocks up to before $N_{i+C-1}$ must have been permuted and the farthest positive block that can be permuted is $P_{i}$, before permuting the first term of the $(i+C)$-th negative block,the first term of the $(i+1)$-th positive block has not yet been selected by our supposition. (3.6) is equivalent to 
\begin{align}
   S^{\sigma}_n \leq S_{q_{i}} + S_{[N_{i+1}, N_{i+C-1}]} = S_{q_{i}} + S_{[N_{i+1}, N_{i+C_0 + M + M*k_1}]}
\end{align}

By (3.4), 
\begin{align}
   S^{\sigma}_n < r
\end{align}
which contradicts (3.5). Therefore, our supposition is false.
\\~\\ \textbf{To prove (2):} we first prove by contradiction that the statement is true for $S$ violating $ST_P$. The proof for the symmetric case proceeds similarly.

Suppose $Z_R$ is not a singleton, then there exists $r_1 > r_2$ where $r_1, r_2 \in Z_R$. By the definition of $Z_R$, we know that there exist some \typeRlambdaPermutation s $\sigma_1, \sigma_2$ of some bounded block number $C_1, C_2$ that fix $S$ to $r_1, r_2$. Let $C=\text{max}(C_1, C_2)$,
\begin{align}
\forall i \in \mathbb{N}, S_{[N_1, N_i]} + S_{[P_1, P_{i-C}]} \leq S^{\sigma_1}_{m_i} \leq  S_{[N_1, N_i]} + S_{[P_1, P_{i+C-1}]}
\end{align}
\begin{align}
\forall j \in \mathbb{N},  S_{[N_1, N_j]} + S_{[P_1, P_{j-C}]} \leq S^{\sigma_2}_{m_j} \leq  S_{[N_1, N_j]} + S_{[P_1, P_{j+C-1}]}
\end{align}

Since $S^{\sigma_1} = r_1$, we have $S^{\sigma_1}_{m_i} \xrightarrow{} r_1, i \xrightarrow{} \infty$ and similarly $S^{\sigma_2}_{m_j} \xrightarrow{} r_2, j \xrightarrow{} \infty$. Let $\epsilon = \frac{(r_1-r_2)}{4}$. Then, by the convergence of a series, the following two inequalities hold.
\begin{align}
\exists i_0 \in \mathbb{N}, \text{ s.t. } \forall i > i_0 , \text{ we have } r_1  - \epsilon < S^{\sigma_1}_{m_i} \leq  S_{[N_1, N_i]} + S_{[P_1, P_{i+C-1}]}, 
\end{align}
\begin{align}
\exists j_0 \in \mathbb{N}, \text{ s.t. } \forall j > j_0, \text{ we have } S_{[N_1, N_j]} + S_{[P_1, P_{j-C}]} \leq S^{\sigma_2}_{m_j} < r_2 + \epsilon
\end{align}
Let $k_0 = \text{max}(i_0, j_0) \in \mathbb{N}$, then we see that (3.11) and (3.12) hold for $\forall i, j > k_0$. By adding these two inequalities and rearranging the terms, we have the following. 
\begin{align}
    S_{[N_1, N_i]} + S_{[P_1, P_{i+C-1}]} - S_{[N_1, N_j]} - S_{[P_1, P_{j-C}]} > r_1 - \epsilon - r_2 - \epsilon = \frac{(r_1-r_2)}{2}, \text{ } \forall i, j > k_0
\end{align}
Let $i = j > k_0$, then we have 
\begin{align}
    S_{[N_1, N_i]} + S_{[P_1, P_{i+C-1}]} - S_{[N_1, N_i]} - S_{[P_1, P_{i-C}]} > r_1 - \epsilon - r_2 - \epsilon = \frac{(r_1-r_2)}{2}, \text{ } \forall i > k_0
\end{align}
\begin{align}
    S_{[P_1, P_{i+C-1}]} - S_{[P_1, P_{i-C}]} = S_{[P_{i-C+1}, P_{i+C-1}]} >  \frac{(r_1-r_2)}{2} > 0, \text{ } \forall i > k_0 \in \mathbb{N}
\end{align}
which contradicts our assumption that $S$ violates $ST_P$.

For the case where we assume that $S$ violates $ST_N$, we adopt the same proof, except that we consider 
\begin{align}
    S_{[N_1, N_j]} + S_{[P_1, P_{j-C}]} - S_{[N_1, N_i]} - S_{[P_1, P_{i+C-1}]} < r_2 + \epsilon - r_1 + \epsilon = - \frac{(r_1-r_2)}{2}, \text{ } \forall i, j > k_0
\end{align}
Let $j = i + 2C -1 > k_0$, then
\begin{align}
  S_{[N_{i+1}, N_{i+2C-1}]} < r_2 + \epsilon - r_1 + \epsilon = -\frac{(r_1-r_2)}{2} < 0, \text{ } \forall i > k_0
\end{align}
which contradicts the assumption that $S$ violates $ST_N$.
\end{proof}
\begin{remark}
Velleman provided \cite{C} examples for $Z_R = \emptyset$ and $Z_R = \mathbb{R}$. Here is an example for $Z_R$ is a singleton:
\begin{multline*}
1 - 1 + \left( \frac{1}{2} - \frac{1}{2} \right) + \left( \frac{1}{3} - \frac{1}{3} \right) + \left( \frac{1}{4} + \frac{1}{4} + \frac{1}{4} + \frac{1}{4} - \frac{1}{4} - \frac{1}{4} - \frac{1}{4} - \frac{1}{4} \right) + \left( \frac{1}{5} - \frac{1}{5} \right) + \cdots \\ 
+ \left( \frac{1}{9} + \frac{1}{9} + \frac{1}{9} + \frac{1}{9} + \frac{1}{9} + \frac{1}{9} + \frac{1}{9} + \frac{1}{9} + \frac{1}{9} - \frac{1}{9} - \frac{1}{9} - \frac{1}{9} - \frac{1}{9} - \frac{1}{9} - \frac{1}{9} - \frac{1}{9} - \frac{1}{9} -\frac{1}{9}\right) \\
+ \left( \frac{1}{10} - \frac{1}{10} \right) + \left( \frac{1}{11} - \frac{1}{11} \right) + \left( \frac{1}{12} - \frac{1}{12} \right) + 
\left( \frac{1}{13} - \frac{1}{13} \right) + \left( \frac{1}{14} - \frac{1}{14} \right)
\cdots
\end{multline*}
The rule of constructing such a series is for positive integer $k$,  if $k \neq n^{2}, \forall n \in \mathbb{N}$,  we add to the series a block $ \left( \frac{1}{k} - \frac{1}{k} \right) $; if  $\exists k \in \mathbb{N}, k=n^{2}$, we add to the series the block $ \left( \frac{1}{k} + \cdots +\frac{1}{k} - \frac{1}{k} - \cdots - \frac{1}{k} \right) $ where the positive and negative fractions repeats $k$ times. It is not difficult to prove that this series is conditionally divergent, has {\typeRlambdaPermutation} with a block number $2$ that fixes it to the value $0$. And this series violates both $ST_P, ST_N$. To see that it cannot be fixed to any other value by a {\typeRlambdaPermutation}, apply the proof of (2) specifically to this series.  
\end{remark}

\section{Future works}
Moving forward from {\typeRlambdaPermutation}, it will be helpful to incrementally relax the constraint it imposes on general $\lambda$-permutation to progress toward the original Velleman's problem. Some ideas include out-of-order selection of elements within a positive/negative block, out-of-order selection of a finite number of same-sign blocks, and combining the above two.

\bibliographystyle{amsplain}

\begin{thebibliography}{10}
\bibitem {A} S. G. Krantz and J. D. McNeal,
 \textit{Creating more convergent series},
The American Mathematical Monthly, \textbf{111}(2004), 32--38.
\bibitem {B} T. M. Apostol, \textit{Calculus}, 2 ed., pp. 413-414, Wiley, Hoboken, 1991.
\bibitem {C} D.J. Velleman, \textit{A Note on $ \lambda $-Permutations}, The American Mathematical Monthly,
\textbf{113} (2006), 173--178.
\end{thebibliography}

\end{document}